\documentclass{article}

\usepackage{amsmath, amssymb}
\usepackage[english]{babel}
\usepackage{paralist}
\usepackage[colorlinks=true]{hyperref}
\usepackage{array}
\usepackage[latin1]{inputenc}
\usepackage{xcolor}
\usepackage{esint}
\usepackage{latexsym,amsfonts}
\usepackage{amsthm}
\usepackage{array}
\usepackage{cleveref}
\usepackage{enumerate}
\usepackage{mathtools}
\usepackage{verbatim}
\usepackage{graphics}

\numberwithin{equation}{section}

\newtheorem{innercustomthm}{Theorem}
\newenvironment{thm-ref}[1]
  {\renewcommand\theinnercustomthm{#1}\innercustomthm}
  {\endinnercustomthm}
\newtheorem{innercustomcor}{Corollary}
\newenvironment{cor-ref}[1]
  {\renewcommand\theinnercustomcor{#1}\innercustomcor}
  {\endinnercustomcor}

\newtheorem{theorem}{Theorem}[section]
\newtheorem{cor}{Corollary}[section]

\newtheorem{lemma}{Lemma}[section]

\newtheorem{prop}{Proposition}[section]

\theoremstyle{definition}

\newtheorem{defn}{Definition}[section]

\newtheorem*{theorem*}{Theorem}

\newtheorem{example}{Example}[section]

\numberwithin{equation}{section}

\newcommand{\C}{\mathbb C} 
\newcommand{\R}{\mathbb R} 
\newcommand{\D}{\mathbb D} 

\def\XXint#1#2#3{{\setbox0=\hbox{$#1{#2#3}{\int}$}
     \vcenter{\hbox{$#2#3$}}\kern-.5\wd0}}

\begin{document}
\title{On the range of harmonic maps in the plane}
\author{Jos\'{e} G. Llorente}
\date{\small{ Departament de Matem\`{a}tiques \\ Universitat Aut\`onoma de Barcelona \\ 08193 Bellaterra.Barcelona \\SPAIN \\
jgllorente@mat.uab.cat }\\ \ \\ \small{\today} }

\maketitle

\let\thefootnote\relax\footnotetext{\hspace{-7pt}\emph{Keywords:}Picard theorem, Liouville theorem, harmonic map,
harmonic function, harmonic polynomial.


Partially supported by Spanish Ministry of Sciences, Innovation and
Universities under grant MTM2017-85666-P, by Generalitat de
Catalunya under grant 2017 SGR 395 and also by the Basque Government
through the BERC 2018-21 program and by Spanish Ministry of
Sciences, Innovation and Universities: BCAM Severo Ochoa
accreditation SEV-2017-0718.}

\begin{abstract}
In 1994 J. Lewis obtained  a purely harmonic proof of the classical
Little Picard Theorem by showing that if the joint value
distribution of two entire harmonic functions satisfies certain
restrictions then they are necessarily constant. We generalize
Lewis' theorem and the harmonic Liouville theorem in terms of the
range of a harmonic map in the plane.

\end{abstract}

\section{Introduction and main results}\label{sec:1}

\

The Little Theorem of Picard says that if an analytic function
defined in the complex plane $\C$ omits two complex values then it
is constant. Since Picard's original proof, based on the modular
function (the universal covering map of $\C \setminus \{ 0,1 \}$),
different proofs have been found, using Bloch's or Schottky's
theorems, normal families and Montel's theorem, curvature of
metrics, the so called Heuristic Bloch Principle and also brownian
motion (see \cite{Re}, \cite{SZ}, \cite{Sch}, \cite{D}, \cite{K}).
Each new proof has contributed in a significative way to broaden and
develop the scope of Geometric Function Theory.

\

During the $80$s and beginning of the $90$s, a number of works were
devoted to generalize Picard's Theorem to real settings. Rickman
(\cite{Ri}) obtained the first version of Picard's Theorem for
quasiregular maps in higher dimensions. Subsequent work of
Eremenko-Sodin (\cite{ES}) and Eremenko-Lewis (\cite{EL}) culminated
in J. Lewis' abridged, purely PDE proof of Rickman's theorem
(\cite{L}). See \cite{BC} for yet another simplification of
Rickman's theorem based in potential-theoretic methods.

\

The general idea behind Lewis approach is roughly the following: if
a finite family of functions belonging to a specific class (let us
say solutions of a PDE) satisfy certain joint value distribution
restrictions then all the functions in the family are constant.
Although Lewis proof is valid for the so called Harnack functions
(including in particular harmonic functions), the method gives an
interesting new proof of the classical Little Picard Theorem.
Indeed, let us assume that $f: \C \to \C \setminus \{ 0,1 \}$ is
analytic. Associated to $f$ there are two natural entire harmonic
functions, namely $u = \log |f|$ and $v = \log |f-1|$. It is a
simple exercise that if $z \in \C \setminus \{0,1 \} $ then
$$
| \log^+ |z| - \log^+ |z-1|\,| \leq \log 2
$$
and
$$
\max \big ( \log |z| , \log |z-1| \big ) \geq - \log 2
$$
Little Picard's Theorem is therefore a consequence of the following
result, which is contained in \cite{L} with more generality (see
also \cite{R}):

\begin{theorem*}[Lewis]\label{lewisthm}
Let $u$, $v: \C \to \R$ be harmonic functions satisfying
\begin{eqnarray}
|u^+ - v^+ | & \leq &  C \vspace{0.2cm} \label{lewisthm1}\\
\max \big (u,v)& \geq  & -C \label{lewisthm2}
\end{eqnarray}
for some constant $C>0$. Then $u$ and $v$ are constant.
\end{theorem*}

\

 The proof of Lewis Theorem relies on two fundamental
steps. Assuming that $u$ is nonconstant, the first step consists of
choosing a sequence of discs at which $u$ exhibits a substantial but
controlled oscillation. This sort of ``signed Harnack" lemma is the
most crucial and technical part of the proof. Secondly, a rescaling
method produces two sequences of harmonic functions in the unit disc
capturing the behavior of $u$ and $v$ in the chosen sequence of
discs. The hypothesis (\ref{lewisthm1}) and (\ref{lewisthm2})
together with well known properties of harmonic functions result
finally in a contradiction. See Theorem 1.3.11 in \cite{R} for
details.

\

Given two entire harmonic functions $u$, $v: \C \to \R$, we will
refer to $f= u+iv: \C \to \C$ as the \textit{harmonic map}
associated to $u$ and $v$. So, for us, an entire harmonic map will
be  just a pair of harmonic functions defined in the complex plane;
in particular no univalence assumption is assumed whatsoever. If $f
= u+iv: \C \to \C$ is a harmonic map we denote its range by
$\mathcal{R}_f = f( \C )$.

\

Our main motivation for the results in this paper was to reinterpret
Lewis theorem in terms of the range $\mathcal{R}_f$. As a first
basic example in this direction, the harmonic Liouville Theorem
(\cite{R}) can be rephrased as follows: if $f= u+iv: \C \to \C$ is a
harmonic map such that  $\mathcal{R}_f$ is contained in a half-plane
then there exist $a$, $b$, $c \in \R$ such that $au + bv = c$. In
particular $\mathcal{R}_f$ is a point or a line.

\

As for Lewis Theorem, it can be read as follows: if $f= u+iv: \C \to
\C$ is a harmonic map and
\begin{equation}\label{lewisrange}
\mathcal{R}_f \subset \{u+iv : \, |u^+ -v^+ |\leq C \, , \, \, \max
(u,v) \geq -C \}\equiv \mathcal{R}_{Le}
\end{equation}
for some constant $C>0$ then $f$ is constant. Observe that the set
in the right hand side of (\ref{lewisrange}) is a cross-like
neigbourhood of the half-lines $\{ u=v \geq 0 \}$, $\{ u=0, v \leq 0
\} $ and $\{ v = 0, u\leq 0 \}$.

\

Before stating our main results we need some definitions and
notation. Hereafter $\D = \{ z\in \C : \, |z| < 1 \}$  will denote
the unit disc in the complex plane and $\partial \D = \{
e^{i\theta}: \, \theta \in \R \}$ the unit circle. If $z_0 \in \C$
and $r>0$, $D(z_0 ,r)$ denotes the (open) disc centered at $z_0$ of
radius $r$.

\begin{defn}[Asymptotic directions]\label{defasymp}
Let $\mathcal{R} \subset \C$. We say that $e^{i\theta} \in \partial
\D$ is an \textit{asymptotic direction} of $\mathcal{R}$ if there
exist a sequence of points $\displaystyle \{ w_n \} \subset
\mathcal{R} $ and a sequence of positive numbers $\displaystyle \{
\varepsilon_n \} $ with $\varepsilon_n \to 0$ as $n\to \infty$ such
that
$$
\lim_{n\to \infty} \varepsilon_n w_n = e^{i\theta}
$$
The \emph{set of asymptotic directions} of $\mathcal{R}$ will be
denoted by $\mathcal{D}(\mathcal{R})$.
\end{defn}

Definition \ref{defasymp} is strongly motivated by the rescaling
method which will be extensively used in  section \ref{sec:4}. Note
also that our concept of asymptotic direction is  broader than the
standard one: if, for instance,  $ \displaystyle \mathcal{R} = \{
u+iv\in \C : \, u\geq v^2 \}$ then the positive $u$-semiaxis is the
only asymptotic direction, that is, $\mathcal{D}( \mathcal{R} ) = \{
1 \}$. Below we include some examples clarifying Definition
\ref{defasymp} in some specific situations. We refer to section
\ref{sec:2} for further properties of asymptotic directions.

\begin{example}\label{lewis example}
If $ \mathcal{R}_{Le}$ is the ``Lewis" range set given by
(\ref{lewisrange}), then $\displaystyle \mathcal{D}
(\mathcal{R}_{Le} ) = \displaystyle \{ e^{i\pi /4}, -1 , -i \} $.
\end{example}

\begin{example}
If $f = u+iv$, where $u\equiv C$ for some constant $C$ and $v$ is
not constant then, by the harmonic Liouville Theorem, $\mathcal{R}_f
= \{ C+i\R \} $ and $\mathcal{D}(\mathcal{R}_f ) = \{i, -i\} $.
Analogously, if $v$ is constant and $u$ is not constant then
$\mathcal{D}(\mathcal{R}_f ) = \{ -1 , 1\}$. If $\mathcal{R}_f$ is
an arbitrary line in the complex plane then
$\mathcal{D}(\mathcal{R}_f )$ would of course reduce to the two
(opposite) directions on $\partial \D$ corresponding to that line.
\end{example}

\begin{example}
Let $f = u+iv$, where $u = x$, $\displaystyle v = e^x \sin y$. Then
$\displaystyle \mathcal{R}_f = \displaystyle \{ u+iv: \, |v| \leq
e^{u} \}$ and $\mathcal{D}(\mathcal{R}_f ) = \displaystyle \{-1 \}
\cup \{ e^{i\theta}: \, -\pi /2 \leq \theta \leq \pi /2  \} $.
\end{example}

\begin{example}
Let $f = u+iv$, where $\displaystyle u = e^{x} \sin y$ and
$\displaystyle v = e^{-x} \sin y$. Then $\displaystyle \mathcal{R}_f
= \displaystyle \{ u+iv: \, 0 < |uv| \leq 1 \} \cup \{ 0 \}$ and
$\displaystyle \mathcal{D}(\mathcal{R}_f ) = \displaystyle \{-1, -i,
1, i\} $.
\end{example}

We describe now the main results of the paper. Recall that two
diametrically opposite points on $\partial \D$ are said
\textit{antipodal}. A closer look at the previous examples shows
that, with the exception of the Lewis range set $\mathcal{R}_{Le}$
(Example \ref{lewis example}), the sets of asymptotic directions
$\mathcal{D}(\mathcal{R}_f )$ in the rest of examples contain pairs
of antipodal points. Our first result says that this must be
actually the case.

\begin{theorem}\label{thm1}
Let $f = u+iv: \C \to \C$ be a harmonic map such that
$\mathcal{D}(\mathcal{R}_f)$  contains no  pair of antipodal points.
Then $f$ is constant.
\end{theorem}

\

The following two theorems go in a slightly different direction: we
discuss assumptions on the range of a harmonic map $f = u+iv$ under
which $u$ and $v$ are lineally dependent. Our next result and the
corollary below extend the harmonic Liouville Theorem.

\begin{theorem}\label{thm2}
Let $f = u+iv: \C \to \C$ be a harmonic map such that
$$
\mathcal{D}(\mathcal{R}_f ) \subset \{ e^{i\theta} : \, \alpha -
\pi/2 \leq \theta \leq \alpha + \pi /2 \}
$$
for some $\alpha \in \R$. Then there exists $c\in \R$ such that
$(\cos \alpha ) u + (\sin \alpha )v = c$. In particular
$\mathcal{R}_f$ is a line or a point.
\end{theorem}

\begin{cor}\label{alpha}
Let $f = u +iv: \C \to \C $ be a harmonic map. Suppose that there
exist $0 \leq \alpha < 1$, $a \geq 0$ and $b\in \R$ such that
$$
v \leq a|u|^{\alpha} + b
$$
Then $v$ is constant. In particular $\mathcal{R}_f$ is a (horizontal
) line or a point.
\end{cor}

Our third result is motivated by the following classical result of
Murdoch(\cite{M}) and Kuran (\cite{Ku}): if $u$ is a nonconstant
harmonic polynomial in $\R^n$, $v$ is  harmonic in $\R^n$ and $u v
\geq 0$ outside a ball, then there is a nonnegative constant
$\lambda$ such that $v = \lambda u$. In particular $v$ is also a
polynomial. Observe that the functions $\displaystyle u = e^x \sin
y$, $v = e^{-x} \sin y$  show that the assumption  that one of the
functions is a polynomial is necessary, even in dimension $2$.

We can restate the Murdoch-Kuran theorem in the plane as follows: if
$f = u+iv: \C \to \C$ is a harmonic map with the additional
assumption that $u$ is a nonconstant harmonic polynomial and $
\displaystyle f(\C \setminus D(0,R)) \subset \{ u+iv: uv\geq 0 \} $
for some $R>0$ then there exists $\lambda \geq 0$ such that $v =
\lambda u$. In particular $v$ is also a polynomial and
$\mathcal{R}_f$ is a line. The next result says that, in the plane,
the union of the two quadrants $\{uv\geq 0 \}$ can be replaced by
any cone symmetric respect to the $u$-axis and having the origin as
a vertex. Note that, even if the aperture of the cone is exactly
$\pi /2$, this is not, a priori, a direct consequence of the
Murdoch-Kuran theorem because a rotation on the $u-v$ plane does not
preserve the fact that one of the functions is a polynomial.
However, we will see that the proof is eventually reduced  to the
case in which both $u$ and $v$ are polynomials.

\begin{theorem}\label{thm3}
Let $f = u+iv: \C \to \C$ be a harmonic map such that $u$ is a
nonconstant harmonic polynomial. Suppose that there exist $a > 0$
and $R>0$ such that
\begin{equation}\label{kuran}
f( \C \setminus D(0,R)) \subset \{ u+iv: \, |u| \leq a |v| \}
\end{equation}
Then there exists $b\in \R$, with $|b| \leq a$ such that $u = bv$.
In particular, $v$ is also a polynomial and $\mathcal{R}_f$ is a
line.
\end{theorem}

\

The structure of the paper is as follows. In section \ref{sec:2} we
show some basic properties of asymptotic directions. Section
\ref{sec:3} reviews Lewis Lemma and a slight generalization that
will be needed later. Section \ref{sec:4} discusses the rescaling
method. Finally, sections \ref{sec:5}, \ref{sec:6} and \ref{sec:7}
are devoted to the proofs of Theorem \ref{thm1}, \ref{thm2} and
\ref{thm3}, respectively.

\

\textit{Acknowledgments.} Part of this research has been done when
the author was visiting the Basque Center for Applied Mathematics
(BCAM). The author wishes to thank this institution, specially
Carlos Pérez for support. He also acknowledges the staff and
researchers at BCAM for the hospitality and stimulating work
atmosphere.

\section{Basic facts about asymptotic directions }\label{sec:2}

We remind that for $\mathcal{R} \subset \C$ then
$\mathcal{D}(\mathcal{R})\subset
\partial \D$ denotes the subset of asymptotic directions associated
to $\mathcal{R}$.

\begin{prop}\label{asym-char}
Let $\mathcal{R} \subset \C$. Then $e^{i\theta} \in
\mathcal{D}(\mathcal{R})$ if and only if there is a sequence $\{ w_n
\} \subset \mathcal{R}$ such that $|w_n | \to \infty $ and
$\displaystyle |w_n |^{-1} w_n  \to e^{i\theta}$ as $n\to \infty$.
\end{prop}
\begin{proof}
Let $e^{i\theta} \in \mathcal{D}(\mathcal{R})$. Choose $\{w_n \}
\subset \mathcal{R}$  and $\varepsilon _n \to 0$ such that
$\varepsilon_n w_n \to e^{i\theta}$ as $n\to \infty$. In particular
$\varepsilon_n |w_n | \to 1$ and
$$
\varepsilon_n w_n - |w_n|^{-1}w_n = (\varepsilon_n |w_n| -1)
|w_n|^{-1}w_n \to 0 \, \, \, \, \, \text{as}\, \, \, \, n\to \infty
$$
which shows that $|w_n|^{-1}w_n \to e^{i\theta}$ as $n\to \infty$.
Conversely, if $\{ w_n \} \subset \mathcal{R}$ such that $|w_n | \to
\infty$ and $|w_n|^{-1}w_n \to e^{i\theta}$ then take $\varepsilon_n
= |w_n |^{-1} \to 0$. Then $\varepsilon_n w_n \to e^{i\theta}$ so
$e^{i\theta} \in \mathcal{D}(\mathcal{R})$.
\end{proof}

\begin{prop}\label{closed} Let $\mathcal{R} \subset \C$.

\begin{enumerate}[a)]

\item If $\mathcal{R}$ is bounded then $\mathcal{D}(\mathcal{R}) =
\varnothing$.
\item If $\mathcal{R}$ is unbounded then $\mathcal{D}(\mathcal{R})$ is a nonempty,
closed subset of $\partial \D$.
\end{enumerate}
\end{prop}

\begin{proof}
Part a) automatically follows from the definition of asymptotic
direction. To prove b), choose a sequence $ \{ w_n \} \subset
\mathcal{R}$ such that $|w_n | \to \infty$ as $n\to \infty$. After
taking a subsequence we may already assume from compactness that
$$
\lim_{n\to \infty} \frac{w_n}{|w_n |} = e^{i\theta}
$$
for some $\theta \in [0, 2\pi)$, which shows that $e^{i\theta} \in
\mathcal{D}(\mathcal{R})$.

\

To see that $\mathcal{D}(\mathcal{R})$ is closed, suppose  that
$e^{i\theta_k} \in \mathcal{D}(\mathcal{R})$ and $ e^{i\theta_k} \to
e^{i\theta}$ as $k\to \infty$, where $\theta_k $, $\theta \in [0,
2\pi )$. By Proposition \ref{asym-char} we can choose $w_k \in
\mathcal{R}$ such that $ | \, |w_k |^{-1} w_k - e^{i\theta_k}| <
1/k$ so $ |w_k |^{-1} w_k \to e^{i\theta} $ as $k\to \infty$. By
Proposition \ref{asym-char}, $e^{i\theta} \in
\mathcal{D}(\mathcal{R})$. Therefore $\mathcal{D}(\mathcal{R})$ is
closed.
\end{proof}

\begin{cor}\label{range-closed}
If $f = u +iv: \C \to \C$ is a nonconstant harmonic map and
$\mathcal{R}_f = f(\C)$ then $\mathcal{D}(\mathcal{R}_f )$ is a
nonempty, closed subset of $\partial \D$.
\end{cor}

\begin{proof}
From the harmonic version of Liouville Theorem,  $\mathcal{R}_f$ is
unbounded. Then apply part b) of Proposition \ref{closed}.
\end{proof}

The following elementary proposition collects some particular
situations.
\begin{prop}\label{examp}
Let $f = u+iv : \C \to \C$ be a harmonic map and $\mathcal{R}_f =
f(\C)$.
\begin{enumerate}[a)]
\item If $\mathcal{R}_f$ is bounded then $f$ is constant and $\mathcal{D}(\mathcal{R}_f) =
\varnothing$.
\item If $u$ is constant and $v$ is not constant then $\mathcal{R}_f$ is a vertical line in the $(u,v)$-plane and
$\mathcal{D}(\mathcal{R}_f) = \{-i, i\}$.
\item If $v$ is constant and $u$ is not constant then $\mathcal{R}_f$ is a horizontal line in the $(u,v)$-plane and
$\mathcal{D}(\mathcal{R}_f) = \{ -1, 1\}$.
\end{enumerate}
\end{prop}

\begin{proof}
Part a) is a direct consequence of part a) in Proposition
\ref{closed}. To prove b) observe first that, by the harmonic
version of Liouville Theorem, $v(\C ) = \R$ so $\mathcal{R}_f$ is a
vertical line in the $(u,v)$-plane. It easily follows  from the
definition of asymptotic direction that $\mathcal{D}(\mathcal{R}_f)
= \{-i, i\}$. Part c) is analogous.
\end{proof}

\begin{prop}\label{asymp}
Let $\mathcal{R} \subset \C $ be unbounded and let
$\mathcal{D}(\mathcal{R}) \subset \partial \D$ be its associated set
of asymptotic directions.

\begin{enumerate}[a)]
\item If $I\subset \partial \D$ is an open arc such that $I\cap \mathcal{D}(\mathcal{R}) =
\varnothing$ then for any closed arc $J\subset I$ there exists
$\rho> 0 $ such that $\displaystyle \mathcal{R} \cap \{ re^{i\theta}
\, : \, r\geq \rho \, , \, e^{i\theta} \in J \} = \varnothing$.
\item If $I\subset \partial\D$ is an open arc and $\displaystyle \mathcal{R} \cap \{ re^{i\theta}
\, : \, r\geq \rho \, , \, e^{i\theta} \in I \} = \varnothing$ for
some $\rho >0$ then $I\cap \mathcal{D}(\mathcal{R}) = \varnothing$.
\end{enumerate}
\end{prop}

\begin{proof}
Suppose that a) does not hold. Then there are sequences $r_n \to
\infty$ and $\displaystyle e^{i\theta_n} \in J$ such that
$\displaystyle w_n = r_n e^{i\theta_n} \in \mathcal{R}$. Since $J$
is closed, there is a subsequence  $\displaystyle e^{i\theta_{n_k}}
\to e^{i\theta}$ as $k\to \infty$, for some $\displaystyle
e^{i\theta} \in J$. Then $\displaystyle |w_{n_k}|^{-1}w_{n_k} \to
e^{i\theta}$, so $\displaystyle e^{i\theta} \in
\mathcal{D}(\mathcal{R})$ by Proposition \ref{asym-char}. This
contradiction proves a).

\

To prove b), suppose that $\displaystyle e^{i\theta} \in I \cap
\mathcal{D}(\mathcal{R})$. By Proposition \ref{asym-char}, choose
$\displaystyle w_n = |w_n|e^{i\theta_n} \in \mathcal{R}$  such that
$|w_n | \to \infty$ and $e^{i\theta_n } \to e^{i\theta}$ as $n\to
\infty$. Since $I$ is open and $|w_n| \to \infty$ then
$\displaystyle e^{i\theta_n} \in I$ and $|w_n| \geq \rho$ for $n$
large enough, so for such $n$'s, $w_n \in \displaystyle \mathcal{R}
\cap \{ re^{i\theta} \, : \, r\geq \rho \, , \, e^{i\theta} \in J
\}$. This contradicts the hypothesis and therefore shows b).
\end{proof}

We will need the following elementary lemma.

\begin{lemma}\label{antipod}
If $ E \subset \partial \D$ is closed and contains no pair of
antipodal points then there is $\alpha \in [0, 2\pi )$ such that $
\displaystyle \{-ie^{i\alpha}, e^{i\alpha}, ie^{i\alpha} \} \subset
\partial \D \setminus E$.
\end{lemma}

\begin{proof}
We prove first that $\partial \D \setminus E$ contains a pair of
antipodal points. Assume that $1\notin E$. If $-1 \notin E$ then we
are done, so suppose that $-1 \in E$. Define
$$
\theta_0 = \sup \{ \theta \in (0, \pi ) \, : (-\theta , \theta )
\subset \partial \D \setminus E \}
$$
Observe that $\theta_0 >0$. If $\theta_0 = \pi $ then $E = \{ -1 \}$
and obviously $\partial D \setminus E$ contains a pair of antipodal
points, so suppose that $0 < \theta_0 < \pi$. Then either
$\displaystyle e^{i\theta_0} \in E$ or $\displaystyle e^{-i\theta_0}
\in E$. Let us assume that $\displaystyle e^{i\theta_0} \in E$, so
\begin{equation*}
\displaystyle \{ e^{i\theta} \, : |\theta | < \theta_0 \} \subset
\partial \D \setminus E
\end{equation*}
By  the hypothesis on $E$, $\displaystyle -e^{i\theta_0 } \notin E$
and, since $E$ is closed, there is $0 < \epsilon < \theta_0$ such
that $\displaystyle -e^{i (\theta_0 -\epsilon )} \notin E$. Then
$\partial \D \setminus E$ contains the pair of antipodal points
$\displaystyle \{ e^{i(\theta_0 - \epsilon )}, -e^{i(\theta_0 -
\epsilon )} \}$. Now, since $E$ contains no pair of antipodal
points, at least one of the points $ie^{i(\theta_0 -\epsilon)}$, $-i
e^{i(\theta_0 -\epsilon )}$ must also lie in $\partial \D \setminus
E$, which completes the proof.
\end{proof}

\

For $\xi = e^{i\alpha} \in \partial \D$ and $0 < \phi < \pi /2$ we
denote by $\displaystyle \mathcal{C}_{\xi ,\phi}$ (resp.
$\displaystyle \mathcal{C}_{\xi , \phi}^+ $) the whole cone (resp.
half cone) with vertex at the origin, axis parallel to $\xi$ and
aperture $2\phi$, that is:
\begin{eqnarray*}
\mathcal{C}_{\xi ,\phi} & = \{ te^{i\theta} \,
 : t\in \R \, , \, |\theta - \alpha | \leq \phi \} \vspace{0.2cm}\\
 \mathcal{C}_{\xi , \phi}^+  & =  \{ te^{i\theta} \, : t\geq 0 \, ,
 \,
|\theta - \alpha | \leq \phi \}
\end{eqnarray*}

\begin{lemma}\label{range-antipod}
Let $f = u +iv : \C \to \C$ be a  harmonic map and $\mathcal{R}_f =
f(\C )$. Suppose that $\mathcal{D}(\mathcal{R}_f)$  contains no pair
of antipodal points. Then there are $\xi \in \partial \D$, $0 < \phi
< \pi /2$ and $\rho >0$ such that
$$
 \mathcal{R}_f \subset  \overline{D}(0, \rho ) \cup  \big (  \C \setminus  (\mathcal{C}_{i\xi ,
 \phi} \cup \mathcal{C}_{\xi , \phi}^+ \cup ) \big)
$$
that is, outside some disc centered at the origin, $\mathcal{R}_f$
does not intersect the union of a whole cone and a half cone having
the origin as a vertex and orthogonal axes.
\end{lemma}

\begin{proof}
By Corollary \ref{range-closed} and Lemma \ref{antipod}, there
exists $ \xi = e^{i\alpha} \in \partial \D$ such that $\displaystyle
\{ -i\xi , \xi , i\xi \}  \subset \partial \D \setminus
\mathcal{D}(\mathcal{R}_f)$. Since $\partial \D \setminus
\mathcal{D}(\mathcal{R}_f)$ is open, the conclusion follows from
part a) of Proposition \ref{asymp}.
\end{proof}

\begin{cor}\label{cones-norm}
Let $g: \C \to \C$ be a harmonic map such that
$\mathcal{D}(\mathcal{R}_g )$ contains no antipodal points. Then
there are $\theta \in [0, 2\pi)$, $\rho >0$ and $a>1$ such that if
$f = e^{i\theta} G$ then
$$
\mathcal{R}_f \subset \overline{D}(0, \rho) \cup \big ( \{ u+iv :
|v| \leq a|u|\} \setminus \{u+iv :  |v| \leq -\frac{1}{a} u  \} \big
)
$$
\end{cor}

\begin{proof}
Let $\xi \in \partial \D$ be associated to $g$ as in Lemma
\ref{range-antipod} and choose $\theta \in [0, 2\pi)$ so that
$e^{i\theta} \xi = -1$. The conclusion readily follows from Lemma
\ref{range-antipod}.
\end{proof}

\section{Lewis Lemma}\label{sec:3}

Lewis' proof of Picard's theorem  is based on a technical lemma
which controls the oscillation of a harmonic function near its zeros
(see \cite{L}, where it was proved for the more general class of
Harnack functions). We state here a version for harmonic functions
in the plane (see \cite{R}, Lemma $1.3.12 $).

\begin{lemma}\label{lewis}
Let $u$ be harmonic in the disc $D(0,R)$ and continuous in the
closed disc $\overline{D}(0,R)$ such that $u(0) = 0$. Then there
exists a disc $D(z,\rho ) \subset D(0,R)$ such that
\begin{eqnarray}
 u(z) & = & 0  \label{zero} \vspace{0.2cm}\\
 M(u,0, R/2 ) & \leq & C_0 \, M(u, z, \rho ) \label{growth}\vspace{0.2cm} \\
 M(u,z,\rho ) & \leq  & C_0 \, M(u,z,\rho /2) \label{doubling}
\end{eqnarray}
for some absolute constant $C_0 >0$.
\end{lemma}

\begin{cor}\label{corolewis}
Let $u$ be a non constant harmonic function in $\C$. Then there
exists a sequence of discs $\{ D(z_n , \rho_n ) \}$ such that, for
all $n$,
\begin{eqnarray}
u(z_n ) & = & 0 \label{zerobis}\vspace{0.2cm} \\
\lim_{n\to \infty }M(u, z_n , \rho_n ) & = &  +\infty \label{growthbis}\vspace{0.2cm}\\
M(u, z_n , \rho_n ) & \leq & C \, M(u, z_n , \rho_n /2)
\label{doublingbis}\vspace{0.3cm}  \label{localbis}
\end{eqnarray}
for some absolute constant $C >0$.
\end{cor}

\begin{proof}
We may assume that $u(0) = 0$. Since $u$ is not constant then $M(u,
0, R) \uparrow +\infty$ as $R \uparrow + \infty$. Take a sequence
$R_n \to +\infty$ and apply Lemma \ref{lewis} to the discs $D(0,
R_n)$. We then get discs $D(z_n, \rho_n)$ such that (\ref{zerobis}),
(\ref{doublingbis}) and (\ref{localbis}) hold. Finally,
(\ref{growthbis}) also holds, since $M(u, 0, R_n /2 ) \leq C \, M(u,
z_n, \rho_n )$ by (\ref{growth}) and $M(u, 0, R_n /2) \to +\infty$
as $n\to \infty$.
\end{proof}

\

For further applications to harmonic maps we will need yet a
refinement of Corollary \ref{corolewis}. We include an elementary
lemma first.

\begin{lemma}\label{abs}
Let $u$ be harmonic in $D(z_0 , r)$ and continuous in
$\overline{D}(z_0, r)$ such that $u(z_0 ) = 0$. Then
$$
M(|u|, z_0 , \frac{2}{3}r) \leq 4 \, M(u, z_0 ,r)
$$
\end{lemma}

\begin{proof}
Put $M = M(u, z_0 , r)$ and let $v = M-u$. Then $v$ is harmonic and
positive in $D(z_0, r)$. Now by Harnack's inequality,
$$
M(v, z_0 , \frac{2}{3}r) \leq 5 v(z_0 ) = 5M
$$
implying that $-4M \leq u \leq M$ in $\displaystyle D(z_0 ,
\frac{2}{3}r)$, so $\displaystyle M(|u|, z_0 , \frac{2}{3}r ) \leq
4M$.

\end{proof}

\

\begin{cor}\label{corolewis2}
Let $u$ be a nonconstant harmonic function in $\C$. Then there
exists a sequence of discs $D(z_n ,r_n)$ such that
\begin{eqnarray}
u(z_n ) & = & 0 \label{zero3}\vspace{0.2cm} \\
\lim_{n\to \infty }M(u, z_n , r_n ) & = &  +\infty \label{growth3}\vspace{0.2cm}\\
M(|u|, z_n , r_n ) & \leq & C_0 \, M(u, z_n , \frac{3}{4}\, r_n )
\label{doubling3}
\end{eqnarray}
for some absolute constant $C_0 >0$.
\end{cor}

\begin{proof}
Let $\{ D(z_n , \rho_n ) \}$ be the sequence of discs provided by
Corollary \ref{corolewis} and set $r_n = \frac{2}{3}\, \rho_n$. Then
(\ref{zero3}) is automatic.  By (\ref{doublingbis}) and Lemma
\ref{abs},
$$
M(|u|, z_n, r_n ) = M(|u|, z_n , \frac{2}{3} \rho_n ) \leq 4 \, M(u,
z_n, \rho_n) \leq 4C \, M(u, z_n, \frac{3}{4}\, r_n)
$$
which proves (\ref{doubling3}).  (\ref{growth3}) is consequence of
(\ref{growthbis}) and the fact that $\displaystyle M(u, z_n , r_n )
= M(u, z_n, \frac{2}{3} \rho_n ) \geq M(u,z_n, \rho_n)$.
\end{proof}

\section{The rescaling method and consequences}\label{sec:4}

Let $f= u+iv: \C \to \C$ be a harmonic map such that
$\mathcal{D}(\mathcal{R}_f )$ contains no pair of antipodal points.
After a rotation we may assume, according to Corollary
(\ref{cones-norm}), that there exist $a>1$ and $\rho >0$ such that
\begin{equation}\label{range}
\mathcal{R}_f \subset \overline{D}(0, \rho) \cup \big ( \{ u+iv :
|v| \leq a|u|\} \setminus \{u+iv :  |v| \leq -\frac{1}{a} u  \} \big
)
\end{equation}
Let $\alpha = \arctan (1/a) \in (0, \pi /4) $. Then it follows from
(\ref{range}) and the definition of asymptotic directions that
\begin{equation}\label{arc1}
\mathcal{D}(\mathcal{R}_f) \subset I_{\alpha}
\end{equation}
where

\begin{equation}\label{arc2}
I_{\alpha } = \{ e^{i\theta}: \theta \in [-\frac{\pi}{2} +\alpha ,
\frac{\pi}{2} -\alpha] \cup [\frac{\pi}{2} +\alpha , \pi -\alpha]
\cup [\pi + \alpha , \frac{3\pi}{2} -\alpha] \}
\end{equation}


 Assume that $f$ is not constant. Then both $u$ and $v$ must
be not constant because otherwise $\mathcal{D}(\mathcal{R}_f)$ would
contain antipodal points, according to Proposition (\ref{examp}).
Now, starting from $u$, let $\{ D(z_n,r_n) \}$ be a sequence of
discs as in Corollary (\ref{corolewis2}) and define the following
two sequences of harmonic functions in $\D$:

$$
u_n (z) = \frac{u(z_n + r_n z )}{M_n} \, , \, \, \, \, \, \, v_n (z)
= \frac{v(z_n +r_n z)}{M_n}
$$
where $M_n = M(|u|, z_n ,r_n )$. From Corollary (\ref{corolewis2})
it follows that
\begin{eqnarray}
u_n (0) & = & 0  \label{zero-norm}\vspace{0.2cm} \\
|u_n| & \leq  & 1 \label{bound-norm} \vspace{0.2cm} \\
M(u_n,0,3/4) & \geq & C_0^{-1}
>0 \label{doubling-norm}
\end{eqnarray}
Also, we get from (\ref{range}) that $|v_n | \leq \max (a,L)$, where
$$
L = \sup_n \frac{M(|v|, 0, \rho )}{M_n}
$$
Observe that $L < \infty$, since $M_n \to +\infty$ as $n\to \infty$.
Then both $\{ u_n \}$ and $\{ v_n \}$ are uniformly bounded
sequences of harmonic functions in $\D$. By Harnack's theorem, there
exists a subsequence $\displaystyle \{ n_k \}$ and harmonic
functions $U$ and $V$ in $\D$ such that
$$
u_{n_k} \to U \, , \, \, \, \, \, \, v_{n_k} \to V
$$
uniformly in compact sets of $\D$. Note that, from
(\ref{zero-norm}), (\ref{bound-norm}) and (\ref{doubling-norm}) it
follows that
\begin{eqnarray}
U(0) & = & 0  \label{zero-norm-lim}\vspace{0.2cm} \\
|U| & \leq  & 1 \label{bound-norm-lim} \vspace{0.2cm} \\
M(U,0,3/4) & \geq & C_0^{-1}
>0 \label{doubling-norm-lim}
\end{eqnarray}
In particular $U$ is nonconstant.

\

\begin{defn}\label{def F}
If $U$ and $V$ are as above, we will call $F = U + iV : \D \to \C$
the \textit{rescaled} harmonic map associated to the original
harmonic map $f = u+iv$. We will denote $\mathcal{R}_F = F( \D )$
the range of $F$.
\end{defn}

The following elementary proposition shows how $\mathcal{R}_F$ is
related to $\mathcal{D}_f$, the set of asymptotic directions of the
original harmonic map $f = u+iv$.

\begin{prop}\label{range-asymp}
If $F = U + iV$ is as above then
$$
\mathcal{R}_F \subset \{re^{i\theta}: r\geq 0 \, , \, e^{i\theta}
\in \mathcal{R}(\mathcal{D}_f) \}
$$
\end{prop}

\begin{proof}
Let $z\in \D$. By definition of $U$ and $V$ there exists a sequence
of complex numbers $\displaystyle \{ z_k \} \subset \C$ and two
sequences of positive numbers $\{ r_k \} $, $\{ M_k \}$ with $M_k
\to +\infty$ as $k\to \infty$ such that if $\displaystyle w_k =
u(z_k +r_k z) + i v(z_k +r_k z)$, then
$$
\lim_{k\to \infty} \frac{w_k}{M_k} = U(z) +iV(z)
$$
so, in particular $|w_k| \to \infty$ as $k\to \infty$. If
$U(z)+iV(z) = 0$ then there is nothing to prove. If $U(z) + iV(z) =
R e^{i\beta} \neq 0$ then we get
$$
\lim_{k\to \infty}\frac{w_k}{|w_k |} = e^{i\beta}
$$
Since $w_k \in \mathcal{R}_f$, the conclusion follows from
Proposition (\ref{asym-char}).
\end{proof}

According to (\ref{arc1}), (\ref{arc2}) we obtain the following
consequence.
\begin{cor}\label{range-asymp2}
If $ F = U +iV$ is as above, then
\begin{equation}\label{range-arcs}
\mathcal{R}_F \subset \{ re^{i\theta}: r\geq 0 \, ,  \, e^{i\theta}
\in I_{\alpha} \}
\end{equation}
where $I_{\alpha}$
 is as in (\ref{arc2}). In particular,
\begin{equation}\label{zeros}
\{ U = 0 \} \subset \{V = 0 \} \subset \{U \geq 0 \}
\end{equation}

\end{cor}

\section{Proof of Theorem \ref{thm1}}\label{sec:5}

We assume from the beginning that  $f = u+iv: \C \to \C$ is a
nonconstant harmonic map (or, equivalently,
$\mathcal{D}(\mathcal{R}_f) \neq \varnothing$ by Proposition
\ref{closed}). The aim of this section is to derive a contradiction
from the further assumption that $\mathcal{D}(\mathcal{R}_f)$ does
not contain any pair of antipodal points. Such a contradiction would
prove Theorem \ref{thm1}. According to section $4$, we may assume
that the rescaled harmonic map $F = U +iV : \D \to \C$ associated to
$f$ as in Definition \ref{def F} satisfies (\ref{range-arcs}) and
(\ref{zeros}).

\

Before starting the proof of Theorem \ref{thm1}, some remarks about
the local structure of the zero set of a harmonic function in the
plane are in order. Suppose that $U$ is a (nonconstant) harmonic
function defined in a neighborhood of the origin in the complex
plane such that $U(0) = 0$. It follows by elementary complex
analysis that there exist $r>0$, an integer $n\geq 1$ (the
multiplicity of the zero) and a conformal map $\phi$ in $D(0,r)$
such that $U(z) = \mathfrak{Re} (\phi (z))^n $ for $z\in D(0,r)$.
This shows in particular that the set
$$
 \{ U = 0 \} \cap D(0,r)
$$
consists of a union of $n$ analytic curves intersecting at the
origin at angle $\displaystyle 2\pi /n$. The complement (in $D(0,r)$
) of such curves is a ``petal-like" region consisting of $2n$
curvilinear sectors meeting at angle $\displaystyle 2\pi /n$ at the
origin such that the sign of $U$ successively alternates in those
sectors. The canonical model is, of course, the function $U =
\mathfrak{Re} (z^n )$.

\

The following lemma is a sort of ``cleaning" result saying that the
inclusion (\ref{range-arcs}) can be locally improved.


\begin{lemma}\label{clean lem} Let $F = U +iV: \D \to \C$ be the rescaled harmonic
map associated to $f = u+iv$, satisfying (\ref{range-arcs}) and
(\ref{zeros}). Then there exists $r>0$ such that
\begin{equation}\label{zeros=}
\{U = 0 \} \cap D(0,r) = \{V = 0 \} \cap D(0,r)
\end{equation}
Furthermore, either $UV \geq 0$ or $UV \leq 0$ in $D(0,r)$. In
particular, either
\begin{equation}\label{clean1}
F(D(0,r)) \subset \{re^{i\theta}: r\geq 0 \, , \, \theta \in [0,
\frac{\pi}{2} - \alpha ] \cup [\pi + \alpha , \frac{3\pi}{2} -
\alpha ] \}
\end{equation}
or
\begin{equation}\label{clean2}
F(D(0,r)) \subset \{ re^{i\theta}: r\geq 0 \, , \, \theta \in [-
\frac{\pi}{2}+ \alpha , 0] \cup [\frac{\pi}{2} + \alpha , \pi -
\alpha ] \}
\end{equation}
for some $\displaystyle 0 < \alpha < \pi /4$.
\end{lemma}

\begin{proof}
Remind that $U$ is not constant and that $U(0) = V(0) = 0$. Then $V$
cannot be constant sine otherwise  $V \equiv 0$, which would imply
$U\geq 0$ by (\ref{zeros}) and therefore $U \equiv 0$ by the Minimum
principle.

\

Choose $r>0$ so that \vspace{-0.2cm}
\begin{eqnarray*}
\{ U = 0 \} \cap D(0,r) & = & \bigcup_{k=1}^n \gamma_k
\vspace{0.2cm} \\
\{ V = 0 \} \cap D(0,r) & = & \bigcup_{j=1}^m \Gamma_j
\end{eqnarray*}
where the $\gamma_k$'s and the $\Gamma_j$'s are analytic curves
meeting at the origin at angles $\displaystyle 2\pi /n$ and $2\pi
/m$, respectively.  Observe that, by (\ref{zeros}), necessarily $n
\leq m$ and each $\gamma_k $ is one of the $\Gamma_j$'s. We claim
that $n = m$ and that both families of curves actually coincide, so
(\ref{zeros=}) follows. The fact that $UV$ has constant sign in
$D(0,r)$ would then be a direct consequence of the local structure
of the (common) zero set of $U$ and $V$ in $D(0,r)$.

To check the claim, suppose that $n<m$. By the above remarks on the
local structure of the zero set, there would be a $j$ such that
$\Gamma_j \setminus \{ 0 \}$ is contained in one of the curvilinear
sectors where $U < 0$, which contradicts (\ref{zeros}). Therefore
$n=m$ and we can assume that $\gamma_k = \Gamma_k$ for $k= 1,
\cdots, n$.
\end{proof}

\

\noindent \textit{Proof of Theorem \ref{thm1}}:

\vspace{0.2cm}

The proof consists of an iterative argument which will result in a
contradiction with the fact that $\mathcal{D}(\mathcal{R}_f)$ does
not contain pairs of antipodal points. Let $F = U +iV$ and $r_1 = r$
be as in Lemma \ref{clean lem}. Assume that $UV \geq 0$ in $D(0, r_1
)$, so (\ref{clean1}) holds. Define \vspace{-0.15cm}

\begin{align*}
\beta_1^{-} & =  \, \inf \{\theta \in (\pi, \frac{3\pi}{2}): \,
\exists R>0 \, \text{such that} \, \, Re^{i\theta} \in F(D(0,r_1))
\}  \vspace{0.2cm}\\
\beta_1^{+} & =  \, \sup \{\theta \in (\pi, \frac{3\pi}{2}): \,
\exists R>0 \, \text{such that} \, \, Re^{i\theta} \in F(D(0,r_1))
\}
\end{align*}
Note that $\displaystyle e^{i\beta_1^{-}}$, $\displaystyle
e^{i\beta_1^{+}} \in \mathcal{D}(\mathcal{R}_f)$ by propositions
\ref{range-asymp} and \ref{closed}. We claim that
\begin{equation}\label{betas}
\pi < \beta_1^{-} < \beta_1^{+} < 3\pi /2
\end{equation}

Observe that the first and third inequalities are consequence of
(\ref{clean1}) and that $\displaystyle \beta_1^{-} \leq \beta_1^{+}$
by definition. If $\beta_1^{-} = \beta_1^{+} = \beta$ then the
intersection of $\displaystyle F(D(0,r_1 ))$ with the third quadrant
would be contained in a line $V = a U$, where $a = \tan \beta >0$.
Then $\displaystyle \{ U < 0 \} \subset \{ V -aU = 0 \} $ in
$D(0,r_1)$ and, by unique continuation and the fact that $U$ takes
negative values near the origin, we would deduce that $V = aU$ in
$\D$ implying that $\mathcal{R}_F = F(\D)$ is a nontrivial line
segment with the origin at its interior. Again by Proposition
\ref{range-asymp}, that would imply that $\displaystyle
\mathcal{D}(\mathcal{R}_f)$ contains the antipodal points
$\displaystyle \pm e^{i\beta}$, which is a contradiction. This
proves (\ref{betas}).

\

Let $\displaystyle a_1^{-} = \tan \beta_1 ^{-}$ and $\displaystyle
a_1^{+} = \tan \beta_1^{+}$. Observe that, by (\ref{zeros=}),
\begin{equation}\label{V-aU 1}
 \{ U = 0 \} \subset \{V -a_1^{-}\, U = 0 \} \cap \{ V
-a_1^{+}\, U = 0 \}
\end{equation}
in $D(0,r_1)$. By the definition of $\displaystyle \beta_1^{-}$ and
$\displaystyle \beta_1^{+}$, we also have (again in $D(0,r_1)$)
that:
\begin{equation}\label{V-aU 2}
\{ V -a_1^{-} U < 0 \} \cap \{ U < 0 \} = \{ V -a_1^{+}U >0 \} \cap
\{U <0 \} = \varnothing
\end{equation}

\

The next step is another ``cleaning" argument applied to the pairs
$U, V - a_1^{-}\, U$ and $U, V -a_1^{+}\, U$. By imitating the proof
of Lemma \ref{clean lem}, it follows from (\ref{V-aU 1}) and
(\ref{V-aU 2}) that we may choose $ 0 < r_2 \leq r_1$ so that
\begin{equation}\label{V-aU 3}
\{ U = 0 \} = \{ V = 0 \} = \{ V - a_1^{-}\, U = 0 \} = \{
V-a_1^{+}\, U = 0 \}
\end{equation}
in $D(0,r_2 )$. Furthermore, $U(V-a_1^{-}\, U)$ and $U(V -a_1^{+}\,
U)$ must have constant signs in $D(0,r_2 )$, which by inspection
turn out to be  $U(V -a_1^{-}\, U ) \leq 0$ and $U(V -a_1^{+}\, U )
\geq 0$. In particular,
\begin{equation}\label{iter1}
F(D(0,r_2)) \subset \{U +iV : (V-a_1^{-}\, U)(V-a_1^{+}\, U) \leq 0
\}
\end{equation}
Since $\displaystyle e^{i\beta_1^{-}}$, $\displaystyle
e^{i\beta_1^{+}} \in \mathcal{D}(\mathcal{R}_f) $ and $\displaystyle
\mathcal{D}(\mathcal{R}_f) $ does not contain pairs of antipodal
points, it follows that $\displaystyle e^{i(\beta_1^{-} - \pi )}$,
$\displaystyle e^{i (\beta_1^{+} - \pi )} \notin
\mathcal{D}(\mathcal{R}_f)$.

\

Now we are ready to run the next step in the iterative argument.
Define
\begin{align*}
\beta_2^{-} & =  \, \sup \{\theta \in (0, \frac{\pi}{2}): \, \exists
R>0 \, \text{such that} \, \, Re^{i\theta} \in F(D(0,r_2))
\}  \vspace{0.2cm}\\
\beta_2^{+} & =  \, \inf \{\theta \in (0, \frac{\pi}{2}): \, \exists
R>0 \, \text{such that} \, \, Re^{i\theta} \in F(D(0,r_2)) \}
\end{align*}
Note that, analogously, $\displaystyle e^{i\beta_2^{-}}$,
$\displaystyle e^{i\beta_2^{+}} \in \mathcal{D}(\mathcal{R}_f)$ and,
since $\displaystyle e^{i(\beta_1^{-} - \pi )}$ and $\displaystyle
e^{i (\beta_1^{+} - \pi )}$ do not belong to the closed set
$\mathcal{D}(\mathcal{R}_f)$ then we get
\begin{equation}\label{iter2}
0 < \beta_1^{-} - \pi < \beta_2^{-} < \beta_2^{+} < \beta_1^{+} -\pi
< \frac{\pi}{2}
\end{equation}
where the third inequality in (\ref{iter2}) follows in the same way
than the second inequality in  (\ref{betas}). By performing another
``cleaning" argument similar to the above we may choose $0 < r_3
\leq r_2$ so that (\ref{iter1}) could be improved to
$$
F(D(0,r_3)) \subset \{ U+iV: (V-a_2^{-}\, U)(V-a_2^{+}\, U) \leq 0
\}
$$
where $\displaystyle a_2^{-} = \tan \beta_2^{-}$ and $\displaystyle
a_2^{+} = \tan \beta_{2}^{+}$. In the next step we would obtain
$\displaystyle \beta_3^{-}$ and $\displaystyle \beta_3^{+}$ so that
$$
\pi + \beta_2^{-} < \beta_3^{-} < \beta_3^{+} < \pi + \beta_2^{+}
$$
and $\displaystyle e^{i\beta_3^{-}}$, $\displaystyle
e^{i\beta_3^{+}} \in \mathcal{D}(\mathcal{R}_f)$. Continuing this
procedure we get sequences $\displaystyle \{ \beta_n^{-} \}$ and
$\displaystyle \{ \beta_{n}^{+} \}$ such that  $\displaystyle \{
\beta_{2n-1}^{-} \}$ and $\displaystyle \{ \beta_{2n}^{-} \}$ are
increasing, $\displaystyle \{ \beta_{2n-1}^{+} \}$ and
$\displaystyle \{ \beta_{2n}^{+} \}$ are decreasing, $\displaystyle
e^{i\beta_n^{-}}$, $\displaystyle e^{i\beta_n^{+}} \in
\mathcal{D}(\mathcal{R}_f)$ and

\[
\begin{array}{c *{5}{@{\;<\;}c}}
 0 &\beta_{2n-1}^{-} - \pi & \beta_{2n}^{-} & \beta_{2n}^{+} & \beta_{2n-1}^{+} - \pi & \frac{\pi}{2}
 \\[0.2cm]
 \pi & \pi + \beta_{2n}^{-} & \beta_{2n+1}^{-} & \beta_{2n+1}^{+} & \pi + \beta_{2n}^{+} &\frac{3\pi}{2}
\end{array}
\]

Then it is easy to check that there are $\displaystyle 0 < \beta^{-}
 \leq \beta^{+}  < \pi /2$ such that
$$
\lim_{n\to \infty}\beta_{2n}^{-}  =  \beta^{-}  \, \, \, , \, \, \,
\lim_{n\to \infty} \beta_{2n}^{+}  =  \beta^{+} ,
$$
and
$$
\lim_{n\to \infty}\beta_{2n-1}^{-} =  \pi + \beta^{-} \, \, \, , \,
\, \,  \lim_{n\to \infty} \beta_{2n-1}^{+} =   \pi + \beta^{+}
$$
Since $\displaystyle \mathcal{D}(\mathcal{R}_f)$ is closed, this
implies in particular that the two antipodal points $\displaystyle
e^{i\beta^{-}}$, $\displaystyle e^{i(\pi + \beta^{-})}$ belong to
$\displaystyle \mathcal{D}(\mathcal{R}_f)$, which is the
contradiction we were seeking for and finishes the proof of the
theorem.

\section{Proof of Theorem \ref{thm2}} \label{sec:6}

Let $f = u+iv : \C \to \C$ be a harmonic map and assume, up to a
rotation, that
\begin{equation}\label{half-space}
 \displaystyle \mathcal{D}(\mathcal{R}_f) \subset \{ e^{i\theta}:
\, -\pi \leq \theta \leq 0 \}
\end{equation}
If $u$ were constant then $\mathcal{R}_f$ would be  a vertical line
or a point in the $u-v$ plane and, according to Proposition
\ref{examp}, either $\mathcal{D}(\mathcal{R}_f) = \varnothing$  or
$\mathcal{D}(\mathcal{R}_f) = \{ -i, i\}$ so Theorem \ref{thm2}
would certainly hold in this case. Assume then that $u$ is not
constant. In particular $u$ is unbounded above and below by the
harmonic Liouville Theorem and, consequently, $\mathcal{R}_f$ cannot
be contained in a half-space of the form $\displaystyle \{u\geq c \}
$ or $\displaystyle \{ u \leq c \}$ for any $c\in \R$.

Now, for each $u\in \R$, let $\displaystyle E_u = \{ v\in \R : \, u
+iv \in \mathcal{R}_f \}$. By the preceding comments, continuity of
$f$ and  connectedness of the set $\mathcal{R}_f = f( \C )$, it
follows that $E_u \neq \varnothing$ for each $u\in \R$. We claim
that $E_u$ is bounded above for any $u\in \R$. Otherwise, we could
find $u\in \R$ and a real sequence $\displaystyle \{ v_n \} $ such
that $\displaystyle v_n \to +\infty$ as $n\to +\infty$ and $u +iv_n
\in \mathcal{R}_f$ for all $n$. Then
$$
\lim_{n\to \infty}\frac{u+iv_n}{|u+iv_n |} = i
$$
which, by Proposition \ref{asym-char}, would imply that $i\in
\mathcal{D}(\mathcal{R}_f) $, contradicting  (\ref{half-space}).
This proves the claim and allows to define $\Phi: \R \to [0,
+\infty)$ by

\begin{equation}\label{def h}
\Phi(u) = \max \big \{ \sup E_u , \, 0 \big \}
\end{equation}

\begin{lemma}\label{pro h}
Let $f= u+iv: \C \to \C$ be a harmonic map satisfying
(\ref{half-space}) and let $\Phi$ be as in (\ref{def h}). Then
$\Phi$ is locally bounded and satisfies
\begin{equation}\label{h asymp}
\lim_{|u|\to \infty} \frac{\Phi(u)}{|u|} = 0
\end{equation}
\end{lemma}

\begin{proof}
Both conclusions are basically consequence of (\ref{half-space}). If
$\Phi$ were not locally bounded, we could find $M>0$ and two real
sequences $\displaystyle \{ u_n \}$, $\displaystyle \{ v_n \}$ such
that $|u_n | \leq M$, $u_n +iv_n \in \mathcal{R}_f$ for all $n$ and
$v_n \to + \infty$ as $n\to \infty$. Then
$$
\lim_{n\to \infty} \frac{u_n +iv_n}{u_n +iv_n} = i
$$
so, again by Proposition \ref{asym-char}, $i\in
\mathcal{D}(\mathcal{R}_f) $, contradicting (\ref{half-space}). The
proof of (\ref{h asymp}) follows similar lines: suppose that (\ref{h
asymp}) does not hold. Then we can find $\alpha >0$ and a sequence
$\displaystyle \{u_n \} $ such that $|u_n | \to \infty$ and
$\Phi(u_n )
> \alpha |u_n |$ for all $n$. Assume that $u_n \to +\infty$. Then
there exist sequences $\displaystyle \{ v_n \}$ and $\displaystyle
\{ \theta_n \}$ such that $v_n \geq \alpha u_n$, \, $\displaystyle
\arctan \alpha \leq \theta_n \leq \pi /2$, \, $u_n +iv_n \in
\mathcal{R}_f$ and
$$
\frac{u_n +iv_n}{|u_n + iv_n |} = e^{i\theta_n}
$$
for all $n$. Choose a subsequence $\displaystyle \{ n_k \}$ such
that  $\displaystyle e^{i\theta_{n_k}} \to e^{i\theta}$ as $k\to
\infty$. Then  $\displaystyle \arctan \alpha \leq \theta \leq \pi
/2$ and, by Proposition \ref{asym-char}, $\displaystyle e^{i\theta}
\in \mathcal{D}(\mathcal{R}_f)$, which contradicts
(\ref{half-space}). The proof of the lemma is now complete.
\end{proof}

Let $f = u+iv: \C \to \C$ be a harmonic map satisfying
(\ref{half-space}) and let $\Phi$ be as in (\ref{def h}). Observe
that the definition of $\Phi$ implies that
\begin{equation}\label{v-Phi}
v(z) \leq \Phi (u(z))
\end{equation}
for each $z\in \C$. Then Theorem \ref{thm2} is consequence of the
following, even more general result.

\begin{theorem}\label{subhar}
Let $w: \C \to \R$ be subharmonic and $u: \C \to \R$ be harmonic.
Suppose that there exists $\Phi : \R \to [0, +\infty )$ such that
$\Phi$ is locally bounded,
\begin{equation}\label{Phi-asy}
\lim_{|t|\to \infty} \frac{\Phi (t)}{|t|} = 0
\end{equation}
and
\begin{equation}\label{w}
w(z) \leq \Phi (u(z))
\end{equation}
for all $z\in \C$. Then $w$ is constant.
\end{theorem}

Observe that Theorem \ref{thm2} follows from Theorem \ref{subhar} by
choosing $w = v$. For the proof of Theorem \ref{subhar} we will need
the following ``relative" Maximum Principle (Theorem 3.1.6 in
\cite{AG}).

\begin{lemma}\label{gardiner}
Let $\Omega \subset \C$ be a domain, $s: \Omega \to \R$ subharmonic
and $h: \Omega \to \R$ harmonic such that $h>0$ in $\Omega$ and
$$
\limsup_{z\to \xi} \frac{s(z)}{h(z)} \leq 0
$$
for each $\xi \in \partial^{\infty} \Omega$. Then $s \leq 0$ in
$\Omega$.
\end{lemma}

\begin{proof}[Proof of Theorem \ref{subhar}]
By the subharmonic Liouville Theorem in the plane (see \cite{R},
Chapter 2) it is enough to prove that $w$ is bounded above. If $u$
is constant then the conclusion follows from (\ref{v-Phi}) so we
assume hereafter that $u$ is not constant, therefore unbounded from
above and below. We will actually show that if $\Omega$ is any
component of $\{ u>0 \}$ or $\{ u < 0 \}$ then $w \leq \Phi (0)$ in
$\Omega$.

Let  $\Omega$ be a component of $\{ u>0 \}$. We pick $z_0 \in \C$
and $r>0$ such that $\overline{D}(z_0, r) \cap \Omega = \varnothing$
(choose $z_0$ and $r$ such that $u<0$ in $\overline{D}(z_0, r)$).
Define
$$
s(z) = w(z) -\Phi (0) \, , \, \, \,  h(z) = u(z) + \log \Big (
\frac{|z-z_0 |}{r} \Big )
$$
and note that $s$ is subharmonic in $\Omega$, $h$ is harmonic in
$\Omega$ and $h>0$ in $\Omega$.

We claim that
\begin{equation}\label{s-h}
\limsup_{z\to \xi} \frac{s(z)}{h(z)} \leq 0 \, \, \, \, \, \,
\text{for all} \, \, \, \, \xi \in \partial^{\infty} \Omega
\end{equation}

To prove the claim we distinguish the cases i) $\xi \in \partial
\Omega \cap \C$ and ii) $\xi = \infty$.

Suppose first that  $\xi  \in
\partial \Omega \cap \C$, in particular $u(\xi) = 0 $. From
subharmonicity and (\ref{w})
\begin{equation}\label{limsup s}
\limsup_{z\to \xi} s(z) \leq s(\xi ) = w(\xi ) -\Phi (0) \leq \Phi
(u(\xi ))- \Phi (0) = 0
\end{equation}
On the other hand,
\begin{equation}\label{liminf h}
\liminf_{z\to \xi} h(z) = \liminf_{z\to \xi} \Big ( u(z) + \log \Big
(\frac{|z-z_0 |}{r} \Big ) \Big ) = \log \Big ( \frac{|z -z_0 |}{r}
\Big )
>0
\end{equation}
so case i)  follows from (\ref{limsup s}) and (\ref{liminf h}).

For case ii) we need to show that
\begin{equation}\label{inf}
\limsup_{\substack{z\to \infty \\ z\in \Omega }} \frac{w(z) - \Phi
(0)}{u(z) + \log \big ( \frac{|z -z_0 |}{r} \big )} \leq 0
\end{equation}
Fix $\varepsilon >0$. From (\ref{Phi-asy}) we can choose $t_0 >0$
such that $\Phi (t) < \varepsilon |t|$ if $|t| \geq t_0$. Let
$$
M_0 = \sup_{[0, t_0]}\Phi
$$ (note that $\Phi$ is locally bounded) and take $|z|$ large enough so that
\begin{equation}\label{choose z}
|z -z_0 | > r \exp{ \big ( \frac{M_0}{\varepsilon} \big )}
\end{equation}

If $u(z)
>t_0$ then
$$
\frac{\Phi (u(z)) -\Phi (0)}{u(z) + \log \big ( \frac{|z -z_0 |}{r}
\big )} \leq \varepsilon
$$
If $0 < u(z) \leq t_0$ then, from (\ref{choose z})
$$
\frac{\Phi (u(z)) -\Phi (0)}{u(z) + \log \big ( \frac{|z -z_0 |}{r}
\big )} \leq \frac{M_0}{\log \big ( \frac{|z -z_0 |}{r} \big )} <
\varepsilon
$$
so case ii) also follows. The hypothesis of Lemma \ref{gardiner} are
then fulfilled and we get $s \leq 0$ in $\Omega$ or, equivalently,
$w \leq \Phi (0)$ in $\Omega$. Since $\Omega$ is arbitrary this
proves that $w \leq \Phi$ in $\{ u>0 \}$. An analogous argument,
replacing $u$ by $-u$ would show that $w \leq \Phi (0)$ in any
component of $\{u < 0 \} $. Then $w \leq \Phi (0)$ in $\C$ and
therefore, $w$ is constant.
\end{proof}

\section{Proof of Theorem \ref{thm3}}\label{sec:7}
Note first that $v$ cannot be constant because in that case
(\ref{kuran}) would imply that $u$ is bounded, therefore constant.
Suppose that $u$ is a harmonic polynomial of degree $n$. By well
known properties of the tracts of harmonic polynomials in the
plane(see \cite{BFHK}) we can choose $R' \geq R$ such that $
\displaystyle \{ u \neq 0 \} \setminus \overline{D}(0, R') =
\bigcup_{i=1}^{2n} G_i$ where each $G_i$ is a ``sector-like"
connected component. On the other hand, assumption (\ref{kuran})
implies
\begin{equation}\label{zero set}
\{ v = 0 \} \cap \big ( \C \setminus D(0, R) \big ) \subset \{ u = 0
\}
\end{equation}

We claim that the set $\{ v \neq 0 \}$ has  a finite number of
components and therefore $v$ is a harmonic polynomial, by Theorem
$1$ in \cite{BFHK}.

It is well known, from the Maximum Principle, that the components of
$\{ v \neq 0 \}$ are nonempty and unbounded. Let $\Omega $ be one
such component and pick $z_0 \in \Omega$ such that $|z_0 | > R'$ and
$u(z_0 )  \neq 0$. Then there is a unique $i\in \{ 1,\cdots, 2n \} $
such that $z_0 \in G_i$ and the connectedness of $G_i$ implies that
$G_i \subset \Omega$. Since the correspondence $\Omega \to G_i$ is
$1-1$ it follows that $\{v \neq 0 \} $ has finitely many components
and therefore $v$ is also a harmonic polynomial.

Consider now the harmonic polynomials $u - av$ and $u +av$ and
observe that $(av -u)(av +u) \geq 0$ outside a disc. By the
Murdoch-Kuran theorem (or perhaps a more elementary argument) we
deduce that $av +u = \lambda (av -u)$ for some $\lambda \geq 0$,
which implies that $u = bv$ for some $b\in \R$ with $|b| \leq a$, as
desired.


\begin{thebibliography}{99}

\bibitem{AG}
\newblock \textsc{D.H. Armitage, S.J. Gardiner},
\newblock \emph{Classical Potential Theory},
\newblock Springer-Verlag, (2001).

\bibitem{BFHK}
\newblock \textsc{D. A. Brannan, W.H.J. Fuchs, W. K. Hayman, Ü.
Kuran},
\newblock \emph{A characterization of harmonic polynomials in the
plane},
\newblock Proc. London Math. Soc., \textbf{(3)32}, (1976), 213-229.

\bibitem{BC}
\newblock \textsc{M. Bonk, P.P. Corradini},
\newblock \emph{The Rickman-Picard theorem}, \\
\newblock  arXiv:1807.07683.



\bibitem{D}
\newblock \textsc{B. Davis},
\newblock \emph{Picard's theorem and brownian motion},
\newblock Trans. Amer. Math. Soc., \textbf{213}, (1975), 353-362.

\bibitem{EL}
\newblock \textsc{A. Eremenko, J. Lewis},
\newblock \emph{Uniform limits of certain A-harmonic functions with applications to quasiregular
mappings},
\newblock Ann. Acad. Sci. Fenn., Ser. A I Math. \textbf{16}, (1991),
361-375.

\bibitem{ES}
\newblock \textsc{A. Eremenko, M. Sodin},
\newblock \emph{Distribution of values of meromorphic functions and meromorphic curves from the standpoint of potential
theory},
\newblock (Russian), Algebra i Analiz \textbf{3}, (1991), no.1,
131-164.


\bibitem{K}
\newblock \textsc{S. Krantz},
\newblock \emph{Complex Analysis: the geometric viewpoint},
\newblock MAA. Carus Mathematical Monographs 23. Second Edition (2004).

\bibitem{Ku}
\newblock \textsc{Ü. Kuran},
\newblock \emph{Generalizations of a theorem on harmonic functions},
\newblock J. London Math. Soc., \textbf{41}, (1966), 145-152.


\bibitem{L}
\newblock \textsc{J. Lewis},
\newblock \emph{Picard's theorem and Rickman's theorem by way of Harnack's inequality},
\newblock Proc. Amer. Math. Soc., \textbf{122}, (1994),no. 1, 199-206.

\bibitem{M}
\newblock \textsc{B. H. Murdoch},
\newblock \emph{A theorem on harmonic functions},
\newblock J. London Math. Soc., \textbf{39}, (1964), 581-588.

\bibitem{R}
\newblock \textsc{T. Randsford},
\newblock \emph{Potential Theory in the complex plane},
\newblock London Mathematical Society Student Texts, (1995).



\bibitem{Re}
\newblock \textsc{R. Remmert},
\newblock \emph{Classical topics in Complex Function Theory},
\newblock Springer-Verlag, (1998).

\bibitem{Ri}
\newblock \textsc{S. Rickman},
\newblock \emph{On the number of omitted values of entire quasiregular mappings},
\newblock J. Analyse Math. \textbf{37}, (1980), 100-117.



\bibitem{Sch}
\newblock \textsc{J. L. Schiff},
\newblock \emph{Normal families},
\newblock Springer-Verlag Universitext, (1993).

\bibitem{SZ}
\newblock \textsc{S. Saks, A. Zygmund},
\newblock \emph{Analytic Functions}. Third Edition,
\newblock Elsevier, (1971).


\end{thebibliography}
\end{document}